\documentclass[12pt,reqno]{amsart}
\usepackage{amsmath, amsfonts, amssymb, amsthm, hyperref}
\usepackage{bm}
\usepackage{mathrsfs}
\allowdisplaybreaks[4]
\def\l{\left}
\def\r{\right}
\def\bg{\bigg}
\def\({\bg(}
\def\){\bg)}
\def\t{\text}
\def\f{\frac}

\def\ord{{\rm ord}}

\def\eq{\equiv}

\def\Z{\mathbb Z}

\def\N{\mathbb N}

\def\p{\mathfrak p}
\def\<{\langle}
\def\>{\rangle}
\def\1{{\bf 1}}

\theoremstyle{plain}
\newtheorem{theorem}{Theorem}

\newtheorem{lemma}{Lemma}

\theoremstyle{definition}

\newtheorem*{Acks}{Acknowledgments}
\theoremstyle{remark}

\numberwithin{equation}{section}

\begin{document}
\title[Some supercongruences concerning truncated hypergeometric series]{Proof of some supercongruences concerning truncated hypergeometric series}
\author[Chen Wang]{Chen Wang}
\address[Chen Wang]{Department of Mathematics, Nanjing
University, Nanjing 210093, People's Republic of China}
\email{cwang@smail.nju.edu.cn}

\author[Dian-Wang Hu]{Dian-Wang Hu*}
\address[Dian-Wang Hu]{School of Mathematics and Statistics, Nanyang Normal University, Nanyang 473061, Henan, People's Republic of China}
\email{hudiwa@163.com}

\begin{abstract}
In this paper, we prove some supercongruences concerning truncated hypergeometric series. For example, we show that for any prime $p>3$ and positive integer $r$,
$$
\sum_{k=0}^{p^r-1}(3k+1)\frac{(\frac12)_k^3}{(1)_k^3}4^k\equiv p^r+\frac76p^{r+3}B_{p-3}\pmod{p^{r+4}}
$$
and
$$
\sum_{k=0}^{(p^r-1)/2}(4k+1)\frac{(\frac12)_k^4}{(1)_k^4}\equiv p^r+\frac76p^{r+3}B_{p-3}\pmod{p^{r+4}},
$$
where $(x)_k=x(x+1)\cdots(x+k-1)$ is the Pochhammer symbol and $B_0,B_1,B_2,\ldots$ are Bernoulli numbers. These two congruences confirm conjectures of Sun [Sci. China Math. 54 (2011), 2509--2535] and Guo [Adv. Appl. Math. 120 (2020), Art. 102078], respectively.
\end{abstract}

\thanks{*Corresponding author}

\maketitle

\section{Introduction}
\setcounter{lemma}{0} \setcounter{theorem}{0}
\setcounter{equation}{0}\setcounter{proposition}{0}
\setcounter{Rem}{0}\setcounter{conjecture}{0}

For $m,n\in\N=\{0,1,2\ldots\}$, the truncated hypergeometric series ${}_{m+1}F_m$ is defined by
$$
{}_{m+1}F_m\bigg[\begin{matrix}x_0&x_1&\ldots&x_m\\ &y_1&\ldots&y_m\end{matrix}\bigg|z\bigg]_n=\sum_{k=0}^n\f{(x_0)_k(x_1)_k\cdots(x_m)_k}{(y_1)_k\cdots(y_m)_k}\cdot\f{z^k}{k!},
$$
where $(x)_k=x(x+1)\cdots(x+k-1)$ is the Pochhammer symbol. During the past few decades, supercongruences concerning truncated hypergeometric series have been widely studied (cf. \cite{GZ,Long2011,MaoZhang2019,Sun2011,Sun2014,vH1997,Wang2020,WangSD2018,Zudilin2009}).

In 2011, Sun \cite{Sun2011} proposed some conjectural supercongruences which relate truncated hypergeometric series to Euler numbers and Bernoulli numbers (see \cite{Sun2011} for the definitions of Euler numbers and Bernoulli numbers). For example, he conjectured that for any prime $p>3$ we have
\begin{equation}\label{sunconj1}
\sum_{k=0}^{(p-1)/2}(3k+1)\f{(\f12)_k^3}{(1)_k^3}4^k\eq p+2(-1)^{(p-1)/2}p^3E_{p-3}\pmod{p^4},
\end{equation}
and for any $r\in\Z^{+}$ we have
\begin{equation}\label{sunconj2}
\sum_{k=0}^{p^r-1}(3k+1)\frac{(\frac12)_k^3}{(1)_k^3}4^k\equiv p^r+\frac76p^{r+3}B_{p-3}\pmod{p^{r+4}},
\end{equation}
where $E_0,E_1,E_2,\ldots$ are Euler numbers and $B_0,B_1,B_2,\ldots$ are Bernoulli numbers. Note that $ak+1=(1+1/a)_k/(1/a)_k$. Thus the sums in \eqref{sunconj1} and \eqref{sunconj2} are actually the truncated hypergeometric series. In 2012, using the WZ method (cf. \cite{PWZ}), Guillera and Zudilin \cite{GZ} proved that
\begin{equation}\label{GZres}
\sum_{k=0}^{p-1}(3k+1)\f{(\f12)_k^3}{(1)_k^3}4^k\eq\sum_{k=0}^{(p-1)/2}(3k+1)\f{(\f12)_k^3}{(1)_k^3}4^k\eq p\pmod{p^3},
\end{equation}
which is \eqref{sunconj2} modulo $p^3$ with $r=1$. In 2019, Mao and Zhang \cite{MaoZhang2019} confirmed \eqref{sunconj1} via a WZ pair found by Guillera and Zudilin \cite{GZ}. The reader is referred to \cite{Sun2019} for further conjectures involving the sums in \eqref{sunconj1} and \eqref{sunconj2}.

Our first theorem confirms \eqref{sunconj2}.
\begin{theorem}\label{mainth1}
For any prime $p>3$ and integer $r\geq1$, we have
\begin{equation}\label{mainth1eq1}
\sum_{k=0}^{p^r-1}(3k+1)\f{(\f12)_k^3}{(1)_k^3}4^k\eq p^r+\f76p^{r+3}B_{p-3}\pmod{p^{r+4}}.
\end{equation}
\end{theorem}

Using the same technique as the one used in the proof of Theorem \ref{mainth1} and using \eqref{sunconj1}, we can also prove that for any prime $p>3$ and positive integer $r$
$$
\sum_{k=0}^{(p^r-1)/2}(3k+1)\f{(\f12)_k^3}{(1)_k^3}4^k\eq p^r+2(-1)^{(p-1)/2}p^{r+2}E_{p-3}\pmod{p^{r+3}}.
$$
It is worth mentioning that Guo and Schlosser \cite{GuoSchlosser2020} obtained two different $q$-analogues of \eqref{GZres}, and similarly to \eqref{GZres}, they conjectured that for any odd prime $p$
$$
\sum_{k=0}^{(p+1)/2}(3k-1)\f{(-\f12)_k^2(\f12)_k}{(1)_k^3}4^k\eq p\pmod{p^3}
$$
which has been confirmed by the first author \cite{Wang2020} by extending it to the modulus $p^4$ case.

In 2011, as a refinement of the (C.2) supercongruence of Van Hamme \cite{vH1997}, Long \cite{Long2011} proved that
\begin{equation}\label{longres}
\sum_{k=0}^{(p-1)/2}(4k+1)\f{(\f12)_k^4}{(1)_k^4}\eq p\pmod{p^4}.
\end{equation}
Guo and Wang \cite{GuoWang2020} obtained a generalization of \eqref{longres}. For any prime $p>3$ and positive integer $r$, they proved that
\begin{equation}\label{GuoWangres}
\sum_{k=0}^{(p^r-1)/2}(4k+1)\f{(\f12)_k^4}{(1)_k^4}\eq p^r\pmod{p^{r+3}}.
\end{equation}

Our next theorem confirms a conjecture of Guo \cite[Conjecture 6.2]{Guo2020} which extends \eqref{GuoWangres} to the modulus $p^{r+4}$ case.

\begin{theorem}\label{mainth2}
Let $p>3$ be a prime and $r$ a positive integer. Then
\begin{equation}\label{mainth2eq1}
\sum_{k=0}^{(p^r-1)/2}(4k+1)\f{(\f12)_k^4}{(1)_k^4}\eq p^r+\f76p^{r+3}B_{p-3}\pmod{p^{r+4}}.
\end{equation}
\end{theorem}

Note that Guo \cite{Guo2020} proved that for any odd prime $p$ and positive integer $r$
$$
\sum_{k=0}^{(p^r-1)/2}(4k+1)\f{(\f12)_k^4}{(1)_k^4}\eq\sum_{k=0}^{p^r-1}(4k+1)\f{(\f12)_k^4}{(1)_k^4}\pmod{p^{r+4}}.
$$

Clearly, the two sums in \eqref{mainth1eq1} and \eqref{mainth2eq1} are the same modulo $p^{r+4}$. Guo \cite[Conjecture 6.3]{Guo2020} conjectured that it is also true for $p=3$.

\begin{theorem}\label{mainth3}
Let $p$ be an odd prime and $r$ a positive integer. Then
\begin{equation}\label{mainth3eq1}
\sum_{k=0}^{p^r-1}(3k+1)\f{(\f12)_k^3}{(1)_k^3}4^k\eq\sum_{k=0}^{(p^r-1)/2}(4k+1)\f{(\f12)_k^4}{(1)_k^4}\pmod{p^{r+4}}.
\end{equation}
\end{theorem}

Note that Guo \cite[Conjecture 6.4]{Guo2020} also conjectured a $q$-analogue of \eqref{mainth3eq1}.

Our main strategy to prove Theorems \ref{mainth1}--\ref{mainth3} is using the WZ method (the reader is referred to \cite{GZ,PWZ,Zudilin2009} for further details and some well known WZ pairs). In fact, the case $r=1$ is easy to deal with since the dominators appearing in the WZ pairs are not divisible by $p$. However, the case $r\geq2$ is very sophisticated. In this case, we need to reduce the sums in \eqref{mainth1eq1} and \eqref{mainth2eq1} to the case $r=1$ via some complicated calculation.

The paper is organized as follows. In both Sections 2 and 3, we shall first establish preliminary results which connect the case $r\geq2$ with the case $r=1$ and play important role in the proof of Theorem \ref{mainth3}. Then we will use the preliminary results to prove Theorems \ref{mainth1} and \ref{mainth2}. In the end of Section 3, we shall give the proof of Theorem \ref{mainth3}.

\section{Proof of Theorem \ref{mainth1}}
\setcounter{lemma}{0} \setcounter{theorem}{0}
\setcounter{equation}{0}\setcounter{proposition}{0}
\setcounter{Rem}{0}\setcounter{conjecture}{0}

We first establish the following result.
\begin{theorem}\label{sec2th}
For any odd prime $p$ and positive integer $r$ we have
$$
\f1{p^r}\sum_{k=0}^{p^r-1}(3k+1)\f{(\f12)_k^3}{(1)_k^3}4^k\eq\f1{p}\sum_{k=0}^{p-1}(3k+1)\f{(\f12)_k^3}{(1)_k^3}4^k\pmod{p^4}.
$$
\end{theorem}

Define the multiple harmonic sum (cf. \cite{Tauraso2018}) as follows:
$$
H_n(s_1,s_2,\ldots,s_r)=\sum_{1\leq k_1< k_2<\cdots<k_r\leq n}\f{1}{k_1^{s_1}k_2^{s_2}\cdots k_r^{s_r}},
$$
where $n\geq r>0$ and each $s_i$ is a positive integer. Multiple harmonic sums have many congruence properties. For example, for any prime $p>s+2$, Sun \cite{SunZH2000} proved that
\begin{equation}\label{harmonic1}
H_{p-1}(s)\eq\begin{cases}\displaystyle-\f{s(s+1)}{2s+4}p^2B_{p-s-2}\pmod{p^3}\ &\t{if}\ 2\nmid s,\vspace{0.2cm}\\
\displaystyle\f{s}{s+1}pB_{p-s-1}\pmod{p^2}\ &\t{if}\ 2\mid s;\end{cases}
\end{equation}
for any $p>5$, Kh. Hessami Pilehrood and T. Hessami Pilehrood \cite[Lemma 3]{Hessami2012} proved that
\begin{equation}\label{harmonic2}
H_{p-1}(1,2)\eq -\f{3H_{p-1}(1)}{p^2}-\f{5H_{p-1}(3)}{12}\pmod{p^3}.
\end{equation}

\begin{lemma}\label{mainth1lem3}
For any odd prime $p$ and positive integer $r$ we have
\begin{gather*}
p^{2r}\sum_{n=1}^{p^r-1}\f{1}{n^2}\binom{2n}{n}\eq p^2\sum_{n=1}^{p-1}\f{1}{n^2}\binom{2n}{n}\pmod{p^4},\\
p^{2r}\sum_{n=1}^{p^r-1}\f{H_{n-1}(1)}{n}\binom{2n}{n}\eq p^2\sum_{n=1}^{p-1}\f{H_{n-1}(1)}{n}\binom{2n}{n}\pmod{p^4},\\
p^{3r}\sum_{n=1}^{p^r-1}\f{H_{n-1}(1)^2}{n}\binom{2n}{n}\eq p^3\sum_{n=1}^{p-1}\f{H_{n-1}(1)^2}{n}\binom{2n}{n}\pmod{p^4},\\
p^{3r}\sum_{n=1}^{p^r-1}\f{H_{n-1}(2)}{n}\binom{2n}{n}\eq p^2\sum_{n=1}^{p-1}\f{H_{n-1}(2)}{n}\binom{2n}{n}\pmod{p^4},\\
p^{3r}\sum_{n=1}^{p^r-1}\f{H_{n-1}(1)}{n^2}\binom{2n}{n}\eq p^2\sum_{n=1}^{p-1}\f{H_{n-1}(1)}{n^2}\binom{2n}{n}\pmod{p^4},\\
p^{2r}\sum_{k=1}^{(p^r-1)/2}\f1{(2k-1)^2}\eq p^{2}\sum_{k=1}^{(p-1)/2}\f1{(2k-1)^2}\pmod{p^4},\\
p^{3r}\sum_{k=1}^{(p^r-1)/2}\f{H_{2k-2}(1)}{(2k-1)^2}\eq p^{3}\sum_{k=1}^{(p-1)/2}\f{H_{2k-2}(1)}{(2k-1)^2}\pmod{p^4},\\
p^{3r}\sum_{k=1}^{(p^r-1)/2}\f{H_{k-1}(1)}{(2k-1)^2}\eq p^{3}\sum_{k=1}^{(p-1)/2}\f{H_{k-1}(1)}{(2k-1)^2}\pmod{p^4}.
\end{gather*}
\end{lemma}

\begin{proof}
We only prove the second congruence, since the other ones can be showed in a similar way. We shall finish the proof by induction on $r$.

Clearly, the second congruence holds for $r=1$. Assume that it holds for $r=k>1$. Now
\begin{align*}
p^{2k+2}\sum_{n=1}^{p^{k+1}-1}\f{H_{n-1}(1)}{n}\binom{2n}{n}=&p^{2k+2}\sum_{\substack{n=1\\p\nmid n}}^{p^{k+1}-1}\f{H_{n-1}(1)}{n}\binom{2n}{n}+p^{2k+2}\sum_{\substack{n=1\\p\mid n}}^{p^{k+1}-1}\f{H_{n-1}(1)}{n}\binom{2n}{n}\\
\eq& p^{2k+2}\sum_{\substack{n=1\\p\mid n}}^{p^{k+1}-1}\f{H_{n-1}(1)}{n}\binom{2n}{n}\\
=&p^{2k+1}\sum_{n=1}^{p^k-1}\f{H_{pn-1}(1)}{n}\binom{2pn}{pn}\pmod{p^4},
\end{align*}
where we used the fact that $\ord_p(H_{n-1}(1))\geq -k$ for $1\leq n\leq p^{k+1}-1$.
Note that
$$
\ord_p(p^{2k+1}/n)\geq k+2\geq4\ \t{for}\ 1\leq n\leq p^k.
$$
Hence for $1\leq n\leq p^k$ we have
$$
\f{p^{2k+1}H_{pn-1}(1)}{n}\eq\f{p^{2k+1}}{n}\sum_{\substack{j=1\\ p\mid j}}^{pn-1}\f1{j}=\f{p^{2k}H_{n-1}(1)}{n}\pmod{p^4}
$$
and
$$
\f{p^{2k}H_{n-1}(1)}{n}\eq0\pmod{p^2}.
$$
By the well-known Kazandzidis congruence (cf. \cite[p. 380]{Robert00}) we have for any odd prime $p$
$$
\binom{2pn}{pn}\eq\binom{2n}{n}\pmod{p^2}.
$$
Combining the above and by the induction hypothesis we arrive at
$$
p^{2k+2}\sum_{n=1}^{p^{k+1}-1}\f{H_{n-1}(1)}{n}\binom{2n}{n}\eq p^{2k}\sum_{n=1}^{p^k-1}\f{H_{n-1}(1)}{n}\binom{2n}{n}\eq p^{2}\sum_{n=1}^{p-1}\f{H_{n-1}(1)}{n}\binom{2n}{n}\pmod{p^4}.
$$
We are done.
\end{proof}

\medskip

\noindent{\it Proof of Theorem \ref{sec2th}}. As in \cite{GZ}, we shall use the following WZ pair
$$
F(n,k)=(3n+2k+1)\f{(\f12)_n(\f12+k)_n^2}{(1)_n^3}4^n
$$
and
$$
G(n,k)=-\f{(\f12)_n(\f12+k)_{n-1}^2}{(1)_{n-1}^3}4^n.
$$
Then we have
$$
F(n,k-1)-F(n,k)=G(n+1,k)-G(n,k)
$$
and
$$
\sum_{k=0}^{p^r-1}(3k+1)\f{(\f12)_k^3}{(1)_k^3}4^k=\sum_{n=0}^{p^r-1}F(n,0).
$$
Clearly,
\begin{align*}
\sum_{n=0}^{p^r-1}F(n,0)=&\sum_{n=0}^{p^r-1}\sum_{k=1}^{(p^r-1)/2}\big(F(n,k-1)-F(n,k)\big)+\sum_{n=0}^{p^r-1}F\l(n,\f{p^r-1}{2}\r)\\
=&\sum_{k=1}^{(p^r-1)/2}\sum_{n=0}^{p^r-1}\big(G(n+1,k)-G(n,k)\big)+\sum_{n=0}^{p^r-1}F\l(n,\f{p^r-1}{2}\r)\\
=&\sum_{k=1}^{(p^r-1)/2}G(p^r,k)+\sum_{n=0}^{p^r-1}F\l(n,\f{p^r-1}{2}\r),
\end{align*}
where the last follows from the fact that $G(0,k)=0$. It suffices to show
\begin{equation}\label{sec2thkey1}
\f1{p^r}\sum_{n=0}^{p^r-1}F\l(n,\f{p^r-1}{2}\r)\eq\f{1}{p}\sum_{n=0}^{p-1}F\l(n,\f{p-1}{2}\r)\pmod{p^4}
\end{equation}
and
\begin{equation}\label{sec2thkey2}
\f{1}{p^r}\sum_{k=1}^{(p^r-1)/2}G(p^r,k)\eq\f{1}{p}\sum_{k=1}^{(p-1)/2}G(p,k)\pmod{p^4}.
\end{equation}

We first consider \eqref{sec2thkey1}. It is easy to see that
$$
F\l(n,\f{p^r-1}{2}\r)=\begin{cases}\displaystyle\f{p^{2r}(3n+p^r)}{4n^2}\binom{2n}{n}\f{(1+\f{p^r}2)_{n-1}^2}{(1)_{n-1}^2}\ &\mbox{if}\ n\geq1,\vspace{0.2cm}\\
\displaystyle p^r\ &\mbox{if}\ n=0.
\end{cases}
$$
Therefore,
$$
\f{1}{p^r}\sum_{n=1}^{p^r-1}F\l(n,\f{p^r-1}{2}\r)=\f{3p^{r}}{4}\sum_{n=1}^{p^r-1}\f{(1+\f{p^r}2)_{n-1}^2}{(1)_{n-1}^2}\cdot\f{\binom{2n}{n}}{n}+\f{p^{2r}}{4}\sum_{n=1}^{p^r-1}\f{(1+\f{p^r}2)_{n-1}^2}{(1)_{n-1}^2}\cdot\f{\binom{2n}{n}}{n^2}.
$$
For $1\leq n\leq p^r-1$, it is clear that $\ord_p(n)\leq r-1$. Note that
$$
\f{(1+\f{p^r}2)_{n-1}}{(1)_{n-1}}=1+\f{p^r}2H_{n-1}(1)+\f{p^{2r}}4H_{n-1}(1,1)+\cdots
$$
and
$$
\ord_p\bigg(H_{n-1}(\overbrace{1,1,\ldots,1}^{d\t{'s}\ 1})\bigg)\geq -d(r-1).
$$
Thus we have
$$
\f{(1+\f{p^r}2)_{n-1}}{(1)_{n-1}}\eq1+\f{p^r}2H_{n-1}(1)+\f{p^{2r}}8H_{n-1}(1)^2-\f{p^{2r}}8H_{n-1}(2)\pmod{p^3},
$$
where we have used $H_{n-1}(1,1)=(H_{n-1}(1)^2-H_{n-1}(2))/2$. Now by Lemma \ref{mainth1lem3},
\begin{align}\label{sec2thkey3}
\f{1}{p^r}\sum_{n=1}^{p^r-1}F\l(n,\f{p^r-1}{2}\r)\eq&\f{3p^{r}}4\sum_{n=1}^{p^r-1}\f{\binom{2n}{n}}{n}\l(1+\f{p^r}2H_{n-1}(1)+\f{p^{2r}}8H_{n-1}(1)^2-\f{p^{2r}}8H_{n-1}(2)\r)^2\notag\\
&+\f{p^{2r}}4\sum_{n=1}^{p^r-1}\f{\binom{2n}{n}}{n^2}\l(1+\f{p^r}2H_{n-1}(1)\r)^2\notag\\
\eq&\f{3p^{r}}4\sum_{n=1}^{p^r-1}\f{\binom{2n}{n}}{n}\l(1+p^rH_{n-1}(1)+\f{p^{2r}}2H_{n-1}(1)^2-\f{p^{2r}}4H_{n-1}(2)\r)\notag\\
&+\f{p^{2r}}4\sum_{n=1}^{p^r-1}\f{\binom{2n}{n}}{n^2}\big(1+p^rH_{n-1}(1)\big)\notag\\
\eq&\f{3p}4\sum_{n=1}^{p-1}\f{\binom{2n}{n}}{n}\l(1+pH_{n-1}(1)+\f{p^2}2H_{n-1}(1)^2-\f{p^2}4H_{n-1}(2)\r)\notag\\
&+\f{p^{2}}4\sum_{n=1}^{p-1}\f{\binom{2n}{n}}{n^2}\big(1+pH_{n-1}(1)\big)\pmod{p^4},
\end{align}
where in the last step we have used the fact (cf. \cite[Theorem 1.3]{SunTauraso2010}) that
\begin{equation}\label{suntauraso}
p^{r-1}\sum_{k=1}^{p^r-1}\f{\binom{2k}{k}}{k}\eq\begin{cases}\displaystyle2\pmod{p^3}\quad&\t{if}\ p=2,\vspace{0.2cm}\\
\displaystyle5\pmod{p^3}\quad&\t{if}\ p=3,\vspace{0.2cm}\\
\displaystyle\f89p^2B_{p-3}\pmod{p^3}\quad&\t{otherwise}.\end{cases}
\end{equation}
Note that in \eqref{sec2thkey3} we actually obtain a result which is independent of $r$. Thus we have proved \eqref{sec2thkey1}.

Now we consider \eqref{sec2thkey2}. Clearly,
$$
\f{1}{p^r}\sum_{k=1}^{(p^r-1)/2}G(p^r,k)= -\f{4p^{2r}}{16^{p^r}}\binom{2p^r}{p^r}^3\sum_{k=1}^{(p^r-1)/2}\f{(\f12+p^r)_{k-1}^2}{(2k-1)^2(\f12)_{k-1}^2}.
$$
By a similar argument as above, $\ord_p(2k-1)\leq r-1$ and
$$
\f{(\f12+p^r)_{k-1}}{(\f12)_{k-1}}\eq 1+2p^rH_{2k-2}(1)-p^rH_{k-1}(1)\pmod{p^2}.
$$
Therefore, by Lemma \ref{mainth1lem3} we have
\begin{align*}
\f{1}{p^r}\sum_{k=1}^{(p^r-1)/2}G(p^r,k)\eq&-\f{4p^{2r}}{16^{p^r}}\binom{2p^r}{p^r}^3\sum_{k=1}^{(p^r-1)/2}\f{1+4p^rH_{2k-2}(1)-2p^rH_{k-1}(1)}{(2k-1)^2}\\
\eq&-\f{4p^2}{16^{p^r}}\binom{2p^r}{p^r}^3\sum_{k=1}^{(p-1)/2}\f{1+4pH_{2k-2}(1)-2pH_{k-1}(1)}{(2k-1)^2}\pmod{p^4}.
\end{align*}
Note that
$$
\sum_{k=1}^{(p-1)/2}\f{1}{(2k-1)^2}=H_{p-1}(2)-\f14H_{(p-1)/2}(2)\eq0\pmod{p}.
$$
Hence we have
$$
\sum_{k=1}^{(p-1)/2}\f{1+4pH_{2k-2}(1)-2pH_{k-1}(1)}{(2k-1)^2}\eq0\pmod{p}.
$$
Then from the Kazandzidis congruence and Fermat's little theorem, we immediately obtain
\begin{equation}\label{sec2thkey4}
\f{1}{p^r}\sum_{k=1}^{(p^r-1)/2}G(p^r,k)\eq-2p^2\binom{2p}{p}^3\sum_{k=1}^{(p-1)/2}\f{1+4pH_{2k-2}(1)-2pH_{k-1}(1)}{(2k-1)^2}\pmod{p^4}.
\end{equation}
This proves \eqref{sec2thkey2} since the right-hand side of the above congruence is independent of $r$.

The proof of Theorem \ref{sec2th} is now complete.\qed

We are now in a position to prove Theorem \ref{mainth1}. We need the following lemmas.

\begin{lemma}\label{mainth1lem1}
For any prime $p>3$ we have
\begin{gather}
\label{mainth1lem1eq1}\sum_{k=1}^{p-1}\f{1}{k^3}\binom{2k}{k}\eq-\f{2H_{p-1}(1)}{p^2}\pmod{p},\\
\label{mainth1lem1eq2}\sum_{k=1}^{p-1}\f{H_k(2)}{k}\binom{2k}{k}\eq \f{2H_{p-1}(1)}{3p^2}\pmod{p},\\
\label{mainth1lem1eq3}\sum_{k=1}^{p-1}\l(\f2{k^2}-\f{3H_k(1)}{k}\r)\binom{2k}{k}\eq \f{2H_{p-1}(1)}{p}\pmod{p^2}.
\end{gather}
\end{lemma}
\begin{proof}
\eqref{mainth1lem1eq1} was originally conjectured by Sun \cite[Conjecture 1.1]{Sun2011} and confirmed by Kh. Hessami Pilehrood and T. Hessami Pilehrood \cite{Hessami2012}. One may consult \cite[Conjecture 1.1]{Sun2011} for the modulus $p^4$ case of \eqref{mainth1lem1eq1}. We can directly verify \eqref{mainth1lem1eq2} and \eqref{mainth1lem1eq3} for $p=5$. By \cite{Tauraso2018} we know these two congruences hold for $p>5$.
\end{proof}

\begin{lemma}\label{mainth1lem2}
For any prime $p>3$ we have
$$
\sum_{k=1}^{p-1}\l(\f{3H_k(1)^2}{k}-\f{4H_k(1)}{k^2}\r)\binom{2k}{k}\eq\f{6H_{p-1}(1)}{p^2}\pmod{p}.
$$
\end{lemma}
\begin{proof}
As in \cite{Tauraso2018}, consider the WZ pair
$$
F(n,k)=\f1{k}\binom{n+k}{k}\quad \t{and}\quad G(n,k)=\f{k}{(n+1)^2}\binom{n+k}{k}.
$$
Then for any $n,k\in\N$ we have
$$
F(n+1,k)-F(n,k)=G(n,k+1)-G(n,k).
$$
Let $S_n=\sum_{k=1}^nF(n,k)H_k(1)^2$. Then
\begin{align}\label{deltaS}
S_{n+1}-S_n=&\sum_{k=1}^{n+1}F(n+1,k)H_k(1)^2-\sum_{k=1}^nF(n,k)H_k(1)^2\notag\\
=&F(n+1,n+1)H_{n+1}(1)^2+\sum_{k=1}^n(F(n+1,k)-F(n,k))H_k(1)^2\notag\\
=&F(n+1,n+1)H_{n+1}(1)^2+\sum_{k=1}^n(G(n,k+1)-G(n,k))H_k(1)^2\notag\\
=&F(n+1,n+1)H_{n+1}(1)^2+\sum_{k=1}^n\bigg(G(n,k+1)H_k(1)^2-G(n,k)H_{k-1}^2\notag\\
&-\f{2G(n,k)H_{k-1}}{k}-\f{G(n,k)}{k^2}\bigg)\notag\\
=&\f{\binom{2n+2}{n+1}H_{n+1}(1)^2}{n+1}+\f{\binom{2n+2}{n+1}H_{n}(1)^2}{2n+2}-2\sum_{k=1}^n\f{G(n,k)H_k(1)}{k}+\sum_{k=1}^n\f{G(n,k)}{k^2}.
\end{align}
By \cite[(16)]{Tauraso2018} we have
\begin{equation}\label{pf1eq1}
\sum_{k=1}^n\f{G(n,k)}{k^2}=\f{1}{(n+1)^2}\sum_{k=1}^n\f1{k}\binom{n+k}{k}=\f{1}{(n+1)^2}\l(\f32\sum_{k=1}^n\f{\binom{2k}{k}}{k}-H_{n}(1)\r).
\end{equation}
From \cite[(1.49)]{G} we know that
$$
\sum_{k=0}^n\binom{x+k}{k}=\binom{x+n+1}{n}.
$$
Therefore,
\begin{align}\label{pf1eq2}
\sum_{k=1}^n\f{G(n,k)H_k(1)}{k}=&\f{1}{(n+1)^2}\sum_{k=1}^n\binom{n+k}{k}\sum_{j=1}^k\f1{j}=\f{1}{(n+1)^2}\sum_{j=1}^n\f1{j}\sum_{k=j}^{n}\binom{n+k}{k}\notag\\
=&\f{1}{(n+1)^2}\sum_{j=1}^n\f1{j}\l(\binom{2n+1}{n}-\binom{n+j}{j-1}\r)\notag\\
=&\f{1}{(n+1)^2}\l(\f12\binom{2n+2}{n+1}H_n(1)-\f{1}{n+1}\sum_{j=1}^n\binom{n+j}{j}\r)\notag\\
=&\f1{2(n+1)^2}\binom{2n+2}{n+1}H_n(1)-\f{1}{2(n+1)^3}\binom{2n+2}{n+1}+\f1{(n+1)^3}.
\end{align}
Substituting \eqref{pf1eq1} and \eqref{pf1eq2} into \eqref{deltaS} we have
\begin{align*}
S_{n+1}-S_n=&\f{3\binom{2n+2}{n+1}H_{n+1}(1)^2}{2n+2}-\f{2\binom{2n+2}{n+1}H_{n+1}(1)}{(n+1)^2}+\f{5\binom{2n+2}{n+1}}{2(n+1)^3}-\f{H_{n+1}(1)}{(n+1)^2}\\
&-\f{1}{(n+1)^3}+\f{3}{2(n+1)^2}\sum_{k=1}^n\f{\binom{2k}{k}}{k}.
\end{align*}
Now summing both sides over $n$ from $0$ to $p-2$ and noting that $S_0=0$ we have
\begin{align}\label{pf1eq3}
S_{p-1}=&\f32\sum_{n=1}^{p-1}\f{\binom{2n}{n}H_n(1)^2}{n}-2\sum_{n=1}^{p-1}\f{\binom{2n}{n}H_n(1)}{n^2}+\f52\sum_{n=1}^{p-1}\f{\binom{2n}{n}}{n^3}-\sum_{k=1}^{p-1}\f{H_n(1)}{n^2}-H_{p-1}(3)\notag\\
&+\f32\sum_{n=1}^{p-1}\f1{n^2}\sum_{k=1}^{n-1}\f{\binom{2k}{k}}{k}.
\end{align}
Clearly, for $1\leq k\leq p-1$,
$$
F(p-1,k)=\f1{k}\binom{p-1+k}{k}\eq\f1{k}\binom{k-1}{k}=0\pmod{p}.
$$
Thus we have
\begin{equation}\label{pf1eq4}
S_{p-1}=\sum_{k=1}^{p-1}F(p-1,k)H_k(1)^2\eq0\pmod{p}.
\end{equation}
In view of \eqref{harmonic1} and \eqref{harmonic2} we have
\begin{equation}\label{pf1eq5}
\sum_{k=1}^{p-1}\f{H_n(1)}{n^2}\eq H_{p-1}(1,2)\eq-\f{3H_{p-1}}{p^2}\pmod{p}.
\end{equation}
Note that
\begin{align}\label{pf1eq6}
\sum_{n=1}^{p-1}\f1{n^2}\sum_{k=1}^{n-1}\f{\binom{2k}{k}}{k}=\sum_{k=1}^{p-1}\f{\binom{2k}{k}}{k}\sum_{n=k}^{p-1}\f{1}{n^2}-\sum_{n=1}^{p-1}\f{\binom{2n}{n}}{n^3}\eq-\sum_{k=1}^{p-1}\f{\binom{2k}{k}H_{k}(2)}{k}\pmod{p}.
\end{align}
Substituting \eqref{pf1eq4}--\eqref{pf1eq6} into \eqref{pf1eq3} and using \eqref{mainth1lem1eq1} and \eqref{mainth1lem1eq2} we immediately obtain the desired result.
\end{proof}

\begin{lemma}\label{harmonickey}
Let $p>3$ be a prime. Then
\begin{gather}
\label{harmonickeyeq1}\sum_{k=1}^{(p-1)/2}\f{1}{(2k-1)^2}\eq\f1{12}pB_{p-3}\pmod{p^2},\\
\label{harmonickeyeq2}\sum_{k=1}^{(p-1)/2}\f{H_{2k-2}(1)}{(2k-1)^2}\eq\f38B_{p-3}\pmod{p},\\
\label{harmonickeyeq3}\sum_{k=1}^{(p-1)/2}\f{H_{k-1}}{(2k-1)^2}\eq\f78B_{p-3}\pmod{p}.
\end{gather}
\end{lemma}
\begin{proof}
By \cite[Corollaries 5.1 and 5.2]{SunZH2000} we have
$$
H_{p-1}(2)\eq \f23pB_{p-3}\pmod{p^2}\quad\t{and}\quad H_{(p-1)/2}(2)\eq \f73pB_{p-3}\pmod{p^2}.
$$
It follows that
$$
\sum_{k=1}^{(p-1)/2}\f{1}{(2k-1)^2}=H_{p-1}(2)-\f14H_{(p-1)/2}(2)\eq \f1{12}pB_{p-3}\pmod{p^2}.
$$
This proves \eqref{harmonickeyeq1}.

Clearly,
\begin{align*}
\sum_{k=1}^{(p-1)/2}\f{H_{2k-2}(1)}{(2k-1)^2}=\sum_{k=1}^{(p-1)/2}\f{H_{2((p+1)/2-k)-2}(1)}{(2((p+1)/2-k)-1)^2}=\f{1}{4}\sum_{k=1}^{(p-1)/2}\f{H_{2k}}{k^2}\pmod{p},
\end{align*}
where we used the fact $H_{p-1-k}(1)\eq H_{k}(1)\pmod{p}$ for any $0\geq k\geq p-1$. In view of \cite[Lemma 2.4]{MaoWang2019}, we have
\begin{equation}\label{maowangres}
\sum_{k=1}^{(p-1)/2}\f{H_{2k}(1)}{k^2}\eq\f32B_{p-3}\pmod{p}.
\end{equation}
This proves \eqref{harmonickeyeq2}.

Note that
$$
H_{(p-1)/2-k}(1)=H_{(p-1)/2}(1)-\sum_{j=1}^k\f{1}{(p+1)/2-j}\eq H_{(p-1)/2}(1)+2H_{2k}(1)-H_k(1)\pmod{p}.
$$
Therefore,
\begin{align*}
&\sum_{k=1}^{(p-1)/2}\f{H_{k-1}(1)}{(2k-1)^2}=\sum_{k=1}^{(p-1)/2}\f{H_{(p-1)/2-k}(1)}{(p-2k)^2}\\
\eq&\f14\sum_{k=1}^{(p-1)/2}\f{H_{(p-1)/2}(1)+2H_{2k}(1)-H_k(1)}{k^2}\\
\eq&\f12\sum_{k=1}^{(p-1)/2}\f{H_{2k}(1)}{k^2}-\f14H_{(p-1)/2}(3)-\f14H_{(p-1)/2}(1,2)\pmod{p}.
\end{align*}
It is known (cf. \cite{Hessami2014}) that
$$
H_{(p-1)/2}(3)\eq-2B_{p-3}\pmod{p}\quad\t{and}\quad H_{(p-1)/2}(1,2)\eq \f32B_{p-3}\pmod{p}.
$$
Then \eqref{harmonickeyeq3} follows from the above two congruences and \eqref{maowangres} immediately.
\end{proof}

\medskip

\noindent{\it Proof of Theorem \ref{mainth1}}. By \eqref{sec2thkey3} and \eqref{sec2thkey4} we have
\begin{align*}
&\f{1}{p^r}\sum_{k=0}^{p^r-1}(3k+1)\f{(\f12)_k^3}{(1)_k^3}4^k\\
\eq&1+\f14p^2\l(3\sum_{n=1}^{p-1}\f{\binom{2n}{n}H_n(1)}{n}-2\sum_{n=1}^{p-1}\f{\binom{2n}{n}}{n^2}\r)+\f18p^3\l(3\sum_{n=1}^{p-1}\f{\binom{2n}{n}H_n(1)^2}{n}-4\sum_{n=1}^{p-1}\f{\binom{2n}{n}H_n(1)}{n^2}\r)\\
&+\f{3}{4}p\sum_{n=1}^{p-1}\f{\binom{2n}{n}}{n}+\f5{16}p^3\sum_{n=1}^{p-1}\f{\binom{2n}{n}}{n^3}-\f3{16}p^3\sum_{n=1}^{p-1}\f{\binom{2n}{n}H_n(2)}{n}-2p^2\sum_{k=1}^{(p-1)/2}\f{1}{(2k-1)^2}\\
&-8p^3\sum_{k=1}^{(p-1)/2}\f{H_{2k-2}(1)}{(2k-1)^2}+4p^3\sum_{k=1}^{(p-1)/2}\f{H_{k-1}(1)}{(2k-1)^2}\pmod{p^4}.
\end{align*}
Then by \eqref{harmonic1}, \eqref{suntauraso} and Lemmas \ref{mainth1lem1}--\ref{harmonickey} we have
$$
\f{1}{p^r}\sum_{n=0}^{p^r-1}(3n+1)\f{(\f12)_n^3}{(1)_n^3}4^n\eq1+\f76p^3B_{p-3}\pmod{p^4}.
$$

The proof of Theorem \ref{mainth1} is now complete.\qed

\section{Proof of Theorem \ref{mainth2}}
\setcounter{lemma}{0} \setcounter{theorem}{0}
\setcounter{equation}{0}\setcounter{proposition}{0}
\setcounter{Rem}{0}\setcounter{conjecture}{0}

Similarly as in Section 2, we first establish the following result.

\begin{theorem}\label{sec3th}
For any odd prime $p$ and positive integer $r$ we have
$$
\f{1}{p^r}\sum_{k=0}^{(p^r-1)/2}(4k+1)\f{(\f12)_k^4}{(1)_k^4}\eq \f{1}{p}\sum_{k=0}^{(p-1)/2}(4k+1)\f{(\f12)_k^4}{(1)_k^4}\pmod{p^4}.
$$
\end{theorem}

\begin{lemma}\label{mainth2lem1}
For any odd prime $p$ and positive integer $r$ we have
\begin{gather}
\label{mainth2lem1eq1}p^rH_{p^r-1}(1)\eq pH_{p-1}(1)\pmod{p^4},\\
\label{mainth2lem1eq2}p^{2r}H_{p^r-1}(1,1)\eq p^2H_{p-1}(1,1)\pmod{p^4},\\
\label{mainth2lem1eq3}p^{3r}H_{p^r-1}(1,1,1)\eq p^3H_{p-1}(1,1,1)\pmod{p^4}.
\end{gather}
\end{lemma}

\begin{proof}
We first prove \eqref{mainth2lem1eq1}. Clearly, it holds for $r=1$. Assume that it holds for $r=k>1$. Then for $r=k+1$ we have
\begin{align*}
p^{k+1}H_{p^{k+1}-1}(1)=&p^{k+1}\sum_{\substack{j=1\\ p\mid j}}^{p^{k+1}-1}\f1{j}+p^{k+1}\sum_{\substack{j=1\\ p\nmid j}}^{p^{k+1}-1}\f1{j}=p^k\sum_{j=1}^{p^k-1}\f1{j}+p^{k+1}\sum_{l=0}^{p^k-1}\sum_{j=1}^{p-1}\f{1}{lp+j}\\
\eq&p^k\sum_{j=1}^{p^k-1}\f1{j}+p^{k+1}\sum_{l=0}^{p^k-1}\sum_{j=1}^{p-1}\l(\f{lp}{j^2}-\f1{j}\r)\\
=& p^kH_{p^k-1}(1)+\f{p^{2k+2}(p^k-1)}{2}H_{p-1}(2)-p^{2k+1}H_{p-1}(1)\\
\eq&p^kH_{p^k-1}(1)\pmod{p^4}.
\end{align*}
By the induction hypothesis we obtain \eqref{mainth2lem1eq1}.

\eqref{mainth2lem1eq2} and \eqref{mainth2lem1eq3} can be proved similarly by noting that
$$
H_{p^r-1}(1,1)=\f{H_{p^r-1}(1)^2-H_{p^r-1}(2)}{2}
$$
and
$$
H_{p^r-1}(1,1,1)=\f{H_{p^r-1}(1)^3-3H_{p^r-1}(1)H_{p^r-1}(2)+2H_{p^r-1}(3)}{6}.
$$
\end{proof}

\begin{lemma}\label{mainth2lem2}
For any odd prime $p$ and positive integer $r$ we have
$$
p^{3r}\sum_{k=0}^{(p^r-3)/2}\f{16^k}{(2k+1)^3\binom{2k}{k}^2}\eq p^{3}\sum_{k=0}^{(p-3)/2}\f{16^k}{(2k+1)^3\binom{2k}{k}^2}\pmod{p^4}.
$$
\end{lemma}

\begin{proof}
Clearly, the congruence holds for $r=1$.

Now suppose that $r\geq2$. In view of the proof of \cite[Theorem 1.1]{PanSun2014}, for $0<k<p^r/2$ we have
\begin{equation}\label{key}
\f{p^r}{k\binom{2k}{k}}\eq0\pmod{p}.
\end{equation}
Thus we obtain
\begin{align*}
p^{3r}\sum_{k=0}^{(p^r-3)/2}\f{16^k}{(2k+1)^3\binom{2k}{k}^2}=&p^{3r}\sum_{\substack{k=0\\ p\mid 2k+1}}^{(p^r-3)/2}\f{16^k}{(2k+1)^3\binom{2k}{k}^2}+p^{3r}\sum_{\substack{k=0\\ p\nmid 2k+1}}^{(p^r-3)/2}\f{16^k}{(2k+1)^3\binom{2k}{k}^2}\\
\eq&p^{3r}\sum_{\substack{k=0\\ p\mid 2k+1}}^{(p^r-3)/2}\f{16^k}{(2k+1)^3\binom{2k}{k}^2}\pmod{p^4},
\end{align*}
since by \eqref{key}, $\ord_p(p^{3r}/\binom{2k}{k}^2)\geq r+2\geq 4$. It is routine to check that
\begin{align*}
p^{3r}\sum_{\substack{k=0\\ p\mid 2k+1}}^{(p^r-3)/2}\f{16^k}{(2k+1)^3\binom{2k}{k}^2}=p^{3r}\sum_{k=0}^{(p^{r-1}-3)/2}\f{4^{(2k+1)p-1}}{(2k+1)^3p^3\binom{(2k+1)p-1}{((2k+1)p-1)/2}^2}.
\end{align*}
Now
\begin{align*}
\binom{(2k+1)p-1}{((2k+1)p-1)/2}=\f{\Gamma((2k+1)p)}{\Gamma(\f{(2k+1)p+1}{2})^2}=-\f{\Gamma_p((2k+1)p)}{\Gamma_p(\f{(2k+1)p+1}{2})^2}\binom{2k}{k},
\end{align*}
where $\Gamma_p(x)$ is the $p$-adic Gamma function (see \cite[Chapter 7]{Robert00} for properties of this function). By \eqref{key}, for $0\leq k\leq (p^{r-1}-3)/2$ we have
$$
\f{p^{3r-3}}{(2k+1)^3\binom{2k}{k}^2}=\f{4p^{3r-3}}{(2k+1)(k+1)^2\binom{2k+2}{k+1}^2}\eq0\pmod{p^3}.
$$
By Fermat's little theorem we have
$$
4^{(2k+1)p-1}=16^{kp}\cdot 4^{p-1}\eq16^k\pmod{p}.
$$
Furthermore,
$$
\f{\Gamma_p((2k+1)p)^2}{\Gamma_p(\f{(2k+1)p+1}{2})^4}\eq\f{\Gamma_p(0)^2}{\Gamma_p(\f12)^4}=1\pmod{p}.
$$
Combining the above we arrive at
\begin{align*}
p^{3r}\sum_{k=0}^{(p^r-3)/2}\f{16^k}{(2k+1)^3\binom{2k}{k}^2}\eq p^{3r-3}\sum_{k=0}^{(p^{r-1}-3)/2}\f{16^k}{(2k+1)^3\binom{2k}{k}^2}\pmod{p^4}.
\end{align*}
Then the desired result follows from induction on $r$.
\end{proof}

\medskip

\noindent{\it Proof of Theorem \ref{sec3th}}. We need the following pair which appeared in \cite{WangSD2018}
\begin{gather*}
F(n,k)=(-1)^k(4n+1)\f{(\f12)_n^3(\f12)_{n+k}}{(1)_n^3(1)_{n-k}(\f12)_k^2},\\
G(n,k)=(-1)^{k-1}\f{4(\f12)_n^3(\f12)_{n+k-1}}{(1)_{n-1}^3(1)_{n-k}(\f12)_k^2}.
\end{gather*}
One may easily check that for any $n\in\N$ and $k\in\Z^{+}$
$$
(2k-1)F(n,k-1)-2kF(n,k)=G(n+1,k)-G(n,k).
$$
Note that such pair is not a WZ pair. Nevertheless, it is also very useful as classical WZ pairs.
For $m\in\N$, by induction on $m$ and noting that
$$
F(n,k)=0\quad\t{for}\ k>n
$$
and
$$
G(0,k)=0\quad\t{for}\ k>0,
$$
we have
\begin{equation}\label{sec3id1}
\sum_{n=0}^{m}F(n,0)=\f{4^m}{\binom{2m}{m}}F(m,m)+\sum_{k=1}^m\f{4^{k-1}G(m+1,k)}{(2k-1)\binom{2k-2}{k-1}}.
\end{equation}
Taking $m=(p^r-1)/2$ in \eqref{sec3id1} we have
\begin{align*}
\sum_{k=0}^{(p^r-1)/2}(4k+1)\f{(\f12)_k^4}{(1)_k^4}=&\sum_{k=0}^{(p^r-1)/2}F(n,0)\\
=&\f{2^{p^r-1}}{\binom{p^r-1}{(p^r-1)/2}}F\l(\f{p^r-1}{2}\r)+\sum_{k=1}^{(p^r-1)/2}\f{4^{k-1}}{(2k-1)\binom{2k-2}{k-1}}G\l(\f{p^r+1}{2},k\r).
\end{align*}
It suffices to show
\begin{equation}\label{sec3thkey1}
\f{1}{p^r}\cdot\f{2^{p^r-1}}{\binom{p^r-1}{(p^r-1)/2}}F\l(\f{p^r-1}{2}\r)\eq \f{1}{p}\cdot\f{2^{p-1}}{\binom{p-1}{(p-1)/2}}F\l(\f{p-1}{2}\r)\pmod{p^4}
\end{equation}
and
\begin{equation}\label{sec3thkey2}
\f{1}{p^r}\cdot\sum_{k=1}^{(p^r-1)/2}\f{4^{k-1}}{(2k-1)\binom{2k-2}{k-1}}G\l(\f{p^r+1}{2},k\r)\eq \f{1}{p}\cdot\sum_{k=1}^{(p-1)/2}\f{4^{k-1}}{(2k-1)\binom{2k-2}{k-1}}G\l(\f{p+1}{2},k\r)\pmod{p^4}.
\end{equation}

We first consider \eqref{sec3thkey1}. Note that
$$
\f{(1/2)_k}{(1)_k}=\f{\binom{2k}{k}}{4^k}.
$$
Thus we have
\begin{align*}
\f{1}{p^r}\cdot\f{2^{p^r-1}}{\binom{p^r-1}{(p^r-1)/2}}F\l(\f{p^r-1}{2}\r)=&\f{1}{p^r}\cdot\f{(-1)^{(p^r-1)/2}2^{p^r-1}(2p^r-1)}{\binom{p^r-1}{(p^r-1)/2}}\cdot\f{(\f12)_{(p^r-1)/2}^3(\f12)_{p^r-1}}{(1)_{(p^r-1)/2}^3(\f12)_{(p^r-1)/2}^2}\\
=&\binom{\f12p^r-1}{p^r-1}\binom{2p^r-1}{p^r-1}.
\end{align*}
For any $p$-adic integer $a$, it is easy to see that
\begin{align}\label{morley}
&\binom{ap^r-1}{p^r-1}=\f{(1+(a-1)p^r)_{p^r-1}}{(1)_{p^r-1}}\notag\\
\eq&1+(a-1)p^rH_{p^r-1}(1)+(a-1)^2p^{2r}H_{p^r-1}(1,1)+(a-1)^3p^{3r}H_{p^r-1}(1,1,1)\notag\\
\eq&1+(a-1)pH_{p-1}(1)+(a-1)^2p^2H_{p-1}(1,1)+(a-1)^3p^3H_{p-1}(1,1,1)\notag\\
\eq&\binom{ap-1}{p-1}\pmod{p^4},
\end{align}
where we have used Lemma \ref{mainth2lem1}. Therefore,
\begin{equation}\label{Fkey}
\f{1}{p^r}\cdot\f{2^{p^r-1}}{\binom{p^r-1}{(p^r-1)/2}}F\l(\f{p^r-1}{2}\r)\eq\binom{\f12p-1}{p-1}\binom{2p-1}{p-1}\pmod{p^4}.
\end{equation}
Then \eqref{sec3thkey1} follows by noting that the right-hand side of the above congruence is independent of $r$.

Below we consider \eqref{sec3thkey2}. It is easy to see that
\begin{align*}
&\f{1}{p^r}\cdot\sum_{k=1}^{(p^r-1)/2}\f{4^{k-1}}{(2k-1)\binom{2k-2}{k-1}}G\l(\f{p^r+1}{2},k\r)\\
=&\f{1}{p^r}\sum_{k=1}^{(p^r-1)/2}\f{4^{k-1}}{(2k-1)\binom{2k-2}{k-1}}\cdot\f{4(-1)^{k-1}(\f12)_{(p^r+1)/2}^3(\f12)_{(p^r-1)/2+k}}{(1)_{(p^r-1)/2}^3(1)_{(p^r+1)/2-k}(\f12)_k^2}.
\end{align*}
Note that
$$
\l(\f12\r)_{(p^r-1)/2+k}=\l(\f12\r)_{(p^r-1)/2}\l(\f{p^r}{2}\r)_k=\f{p^r}{2}\l(\f12\r)_{(p^r-1)/2}\l(1+\f{p^r}{2}\r)_{k-1}
$$
and
$$
(1)_{(p^r+1)/2-k}=(-1)^{k-1}\f{(1)_{(p^r-1)/2}}{(\f12-\f{p^r}2)_{k-1}}.
$$
Therefore,
\begin{align*}
&\f{1}{p^r}\cdot\sum_{k=1}^{(p^r-1)/2}\f{4^{k-1}}{(2k-1)\binom{2k-2}{k-1}}G\l(\f{p^r+1}{2},k\r)\\
=& \f{p^{3r}}{16^{p^r-1}}\binom{p^r-1}{(p^r-1)/2}^4\sum_{k=0}^{(p^r-3)/2}\f{64^k(1+\f{p^r}2)_k(\f12-\f{p^r}{2})_k}{(2k+1)^3\binom{2k}{k}^3(1)_k^2}.
\end{align*}
By \eqref{key} we have
$$
\f{p^{3r}}{(2k+1)^3\binom{2k}{k}^3}=\f{8p^{3r}}{(k+1)^3\binom{2k+2}{k+1}^3}\eq0\pmod{p^3}.
$$
Thus by Fermat's little theorem, \eqref{morley} and Lemma \ref{mainth2lem2} we further obtain
\begin{align}\label{Gkey}
\f{1}{p^r}\cdot\sum_{k=1}^{(p^r-1)/2}\f{4^{k-1}}{(2k-1)\binom{2k-2}{k-1}}G\l(\f{p^r+1}{2},k\r)\eq& p^{3r}\sum_{k=0}^{(p^r-3)/2}\f{16^k}{(2k+1)^3\binom{2k}{k}^2}\notag\\
\eq& p^3\sum_{k=0}^{(p-3)/2}\f{16^k}{(2k+1)^3\binom{2k}{k}^2}\pmod{p^4}.
\end{align}
This proves \eqref{sec3thkey2}.

The proof of Theorem \ref{sec3th} is now complete.
\qed

\medskip

\noindent{\it Proof of Theorem \ref{mainth2}}. By \eqref{Fkey} and \eqref{Gkey} we arrive at
\begin{equation}\label{mainth2key1}
\f{1}{p^r}\sum_{k=0}^{(p^r-1)/2}(4k+1)\f{(\f12)_k^4}{(1)_k^4}\eq\binom{\f12p-1}{p-1}\binom{2p-1}{p-1}+p^3\sum_{k=0}^{(p-3)/2}\f{16^k}{(2k+1)^3\binom{2k}{k}^2}\pmod{p^4}.
\end{equation}
By \eqref{morley},
$$
\binom{\f12p-1}{p-1}\eq 1-\f12pH_{p-1}(1)+\f14p^2H_{p-1}(1,1)-\f18p^3H_{p-1}(1,1,1)\pmod{p^4}
$$
and
$$
\binom{\f2p-1}{p-1}\eq 1+pH_{p-1}(1)+p^2H_{p-1}(1,1)+p^3H_{p-1}(1,1,1)\pmod{p^4}.
$$
It is known (cf. \cite{Hessami2014}) that
$$
H_{p-1}(1,1)\eq -\f13pB_{p-3}\pmod{p^2}
$$
and
$$
H_{p-1}(1,1,1)\eq0\pmod{p}.
$$
Then in view of \eqref{harmonic1} we get that
\begin{equation}\label{mainth2key2}
\binom{\f12p-1}{p-1}\binom{2p-1}{p-1}\eq 1-\f7{12}p^3B_{p-3}\pmod{p^4}.
\end{equation}
Sun \cite[Theorem 1.2]{Sun2014} proved that
\begin{equation}\label{mainth2key3}
\sum_{k=0}^{(p-3)/2}\f{16^k}{(2k+1)^3\binom{2k}{k}^2}\eq\f74B_{p-3}\pmod{p}.
\end{equation}
Substituting \eqref{mainth2key2} and \eqref{mainth2key3} into \eqref{mainth2key1} we immediately obtain Theorem \ref{mainth2}.\qed

\medskip

Now we can easily obtain Theorem \ref{mainth3}.
\medskip

\noindent{\it Proof of Theorem \ref{mainth3}}. The case $p>3$ is the immediate corollary of Theorems \ref{mainth1} and \ref{mainth2}.

Now we consider the case $p=3$. In view of Theorems \ref{sec2th} and \ref{sec3th}, we only need to prove \eqref{mainth3} for $r=1$. In fact, if $p=3$ and $r=1$, one may check \eqref{mainth3} directly.\qed

\begin{Acks}
The authors are grateful to Prof. Zhi-Wei Sun for his helpful suggestions on this paper. This work was supported by the National Natural Science Foundation of China (grant no. 11971222)
\end{Acks}

\end{document}